\documentclass[11pt]{article}

\usepackage{geometry}

\usepackage{tikz}
\tikzset{every picture/.style={scale=0.5}}

\usepackage{graphicx,float,color,subcaption,caption}

\usepackage{enumitem}

\usepackage{authblk}

\usepackage{amssymb,amsmath,amsthm,mathtools}

\usepackage[linktoc=all,hidelinks]{hyperref}

\DeclareMathOperator{\supp}{supp}
\DeclareMathOperator{\ord}{ord}

\newcommand{\bra}[1]{\langle#1\rangle}
\newcommand{\bras}[1]{\langle#1\rangle_s}

\newcommand{\G}[1]{\mathbb Z_{#1}\times\mathbb Z_{#1}}

\newcommand{\EF}[1]{\widehat{\mathbf 1}_{#1}}
\newcommand{\F}[1]{\EF{#1}^{sym}}
\newcommand{\EZ}[1]{Z(\EF{#1})}
\newcommand{\Z}[1]{Z(\F{#1})}

\newtheorem{Thm}{Theorem}
\newtheorem*{Thm*}{Theorem}
\newtheorem{Lm}{Lemma}
\newtheorem*{Lm*}{Lemma}
\newtheorem{Cor}{Corollary}
\newtheorem{Prop}{Proposition}

\title{On Tiling and Spectral Sets in $\mathbb Z_{p^2}\times\mathbb Z_{p^2}$}
\author{Weiqi Zhou\thanks{zwq@xzit.edu.cn}}
\affil{\small School of Mathematics and Statistics, Xuzhou University of Technology \\  {\footnotesize Lishui Road 2, Yunlong District, Xuzhou, Jiangsu Province, China 221111}}
\date{}							

\begin{document}
\maketitle
\begin{abstract}
Let $p$ be a prime number, it is shown that tiling sets are spectral in $\mathbb Z_{p^2}\times\mathbb Z_{p^2}$ and spectral sets of sizes $p^k (k=0,1,2,3,4)$ are tiling in $\mathbb Z_{p^2}\times\mathbb Z_{p^2}$. The main arguments are based on analysis of the zero set of the Fourier transform as well as the tree structure of subgroups in $\mathbb Z_{p^2}\times\mathbb Z_{p^2}$. Some auxiliary results concerning tiling sets and spectral sets of sizes $p$ and $p^{2m-1}$ in $\G{p^m}$ ($m\in\mathbb N$) are also presented. \\

{\noindent
}
\end{abstract}
 
\section{Introduction} 
Let $A,B$ be subsets of an additive finite Abelian group $G$, we use $A+B$ for the multiset (a \emph{multiset} is a collection of elements in which elements are allowed to be repeated) formed by elements of form $a+b$ where $a,b$ are enumerated from $A,B$ respectively. We write $A\oplus B$ if $A+B$ is actually a usual set (i.e., no elements appears more than once), and in this case we may also say that $A,B$ are tiling complements of each other in $A\oplus B$. A subset $A$ is said to be a \emph{tiling set} (or \emph{a tile}) in $G$, if there exists another subset $B$ of $G$ so that $A\oplus B=G$, in which case we shall call $(A,B)$ a tiling pair in $G$.

Let $\mathbb Z_n=\{0,1,\ldots,n-1\}$ be the additive cyclic group of $n$ elements, a subset $A$ is said to be \emph{spectral} in $\mathbb Z_n\times\mathbb Z_n$, if there is some $S\subseteq\G{n}$ such that $\{e^{2\pi i\bra{x,s}/n}\}_{s\in S}$ (if $x=(x_1,x_2)$ and $s=(s_1,s_2)$, then $\bra{x,s}=x_1s_1+x_2s_2$ is the inner product) is an orthogonal bases on $L^2(A)$ with respect to the counting measure. In such case we shall call $(A,S)$ a \emph{spectral pair} in $\G{n}$, and $S$ is said to be a \emph{spectrum} of $A$.

A fundamental question of the sampling theory, known as the Fuglede's conjecture \cite{fuglede1974}, asks whether these two types of sets coincide. The conjecture has been disproved in $\mathbb R^d$ for $d\ge 3$ \cite{farkas2006, matolcsi2006, kolounzakis2006, matolcsi2005, tao2004} decades ago, counterexamples for $d=2$ are also constructed in \cite{Zhang2026} recently. For $d=1$ it remains open (see \cite{lev2022} for positive results on convex sets in all dimensions). The same conjecture in finite Abelian groups is equally important and equivalent to the conjecture in $\mathbb R$ \cite{dutkay2014, fu2026, iosevich2013, wang1996}. On $p$-adic fields it is known to be true \cite{fan2016,fan2019}. 

So far there are quite a number of positive results in cyclic groups, in which the method evolves: Lagarias and Wang showed that it is true in periodic cyclic groups (i.e., if $(A,B)$ is a tiling pair in such a group, then one of $A,B$ has to be periodic) \cite{wang1997}, which includes all cyclic $p$-groups. Instead of looking at periodic components or other weakened versions of periodicity, Coven and Meyerowitz proposed two different conditions to characterize tiles in cyclic groups $\mathbb Z_n$, and they proved its sufficiency for all $n$ and necessity when $n$ is the product of two prime powers \cite{coven1999}. \L{}aba showed that a set satisfying these conditions will also admit a spectrum \cite{laba2002}. Therefore to show that a set is tiling/spectral in cyclic groups, we may verify whether they satisfy the Coven-Meyerowitz condition, this approach has been fruitful \cite{malikiosis2020, laba2022, laba2024, malikiosis2022, somlai2019, zhang2024}, and has led to several positive results in $\mathbb Z_n$ including $n$ being square-free or product of two prime powers. 

For groups generated by two elements the number of positive results are limited \cite{fallon2022, iosevich2017, shi2020, zhang2021}. Despite the counterexamples, in \cite{shi2020} it is asked whether all tiles are spectral in $\mathbb Z_{p^n}^d$ ($p^n$ is a prime power), since in $p$-groups so far there has been no non-spectral tile found, and the question is answered positively at least for certain cases in higher dimensions \cite{fan2024, malikiosis2024}. The tiling to spectral part of this paper is a critical case left open in \cite{fan2024}.

The main result of this paper is

\begin{Thm} \label{ThmMain}
Let $p$ be a prime number, then 
\begin{enumerate}[leftmargin=*, label=(\roman*)]
\item tiling sets are spectral in $\G{p^2}$.
\item spectral sets of sizes $1,p,p^2,p^3,p^4$ are tiling in $\G{p^2}$.
\end{enumerate}
\end{Thm}

The paper is organized in the following way: In Section \ref{SecPre} we introduce two types of Fourier transforms on $\G{n}$ and motivates the usage of the symplectic Fourier transform in this particular problem. In Section \ref{SecPm} we present auxiliary results (Lemma \ref{LmTSp}, Lemma \ref{LmTSq}, Lemma \ref{LmSTq}, Lemma \ref{LmSTp}) concerning tiling and spectral sets of sizes $p$ and $p^{2m-1}$ in $\G{p^m}$ ($m\in\mathbb N$). In Section \ref{SecP2} we analyze the structure of $\G{p^2}$ and zeros of Fourier transforms on $\G{p^2}$. Finally in Section \ref{SecMain} we prove Theorem \ref{ThmMain}.

\section{Preliminaries} \label{SecPre}
\subsection{The Euclidean setting} 
Let $A$ be a subset of a finite Abelian group $G$, define its \emph{difference set} to be
$$\Delta A=\{a-a': a,a'\in A, a\neq a'\}.$$

\begin{Lm} \label{LmDelta}
Suppose $A,B$ are subsets of some finite Abelian group $G$, if
$$|A|\cdot |B|\ge|G|, \quad\text{and}\quad \Delta A\cap \Delta B=\emptyset,$$
then we must have $|A|\cdot |B|=|G|$ and that $(A,B)$ is a tiling pair in $G$.
\end{Lm}

\begin{proof}
Suppose $a,a'\in A, b,b'\in B$, then
$$a+b=a'+b' \quad\Leftrightarrow\quad a-a'=b-b',$$
indicates
\begin{equation} \label{EqDiff}
A\oplus B \text{ is well defined} \quad\Leftrightarrow\quad \Delta A\cap \Delta B=\emptyset.
\end{equation}
Clearly $A\oplus B$ is a subset of $G$, thus 
$$|A\oplus B|\le |G|.$$
On the other hand by the given assumption we also have
$$|A\oplus B|=|A|\cdot|B|\ge|G|.$$
Together this means
$$|A\oplus B|=|G|,$$
which indicates that $(A,B)$ is a tiling pair in $G$, therefore $|A|\cdot |B|=|G|$ must also hold.
\end{proof}

Let $f$ be a function on $\G{n}$, denote by  
$$\supp(f)=\{x\in\G{n}: f(x)\neq 0\}, \quad Z(f)=\{x\in\G{n}: f(x)=0\},$$
its support set and its zero set respectively. Its \emph{Fourier transform} is defined as
$$\widehat{f}(\xi)=\sum_{x\in\G{n}}f(x)e^{2\pi i\bra{x,\xi}/n},\quad \xi\in\G{n}.$$
In particular, if $f=\mathbf 1_A$ is the characteristic function for some $A\subset\G{n}$, then we get
\begin{equation} \label{EqDefF}
\EF{A}(\xi)=\sum_{a\in A}e^{2\pi i\bra{a,\xi}/n}.
\end{equation}
From this definition we can verify directly that
\begin{equation} \label{EqConv}
\EF{A+B}=\EF{A}\cdot\EF{B}.
\end{equation}
It is thus clear that $(A,B)$ being a tiling pair in $\G{n}$ if and only if
$$\EF{A}\cdot\EF{B}=\EF{\G{n}}.$$
It is also easy to verify that 
$$\EZ{\G{n}}=(\G{n})\setminus\{(0,0)\},$$
therefore $(A,B)$ being a tiling pair in $\G{n}$ is also equivalent to 
$$|A|\cdot|B|=|\G{n}| \quad\text{and}\quad \Delta(\G{n})=(\G{n})\setminus\{(0,0)\}\subseteq Z(\EF{A}\cdot\EF{B}).$$
Similarly $(A,S)$ being a spectral pair is equivalent to
$$|A|=|S| \quad\text{and}\quad \Delta S\subseteq \EZ{A}.$$

Denote by $\ord(\xi)$ the oder of an element $\xi\in\G{n}$. Given $A,\xi$, to understand whether $\xi\in \EZ{A}$, we first notice that there must be a unique $h\in\G{n}$ so that $\xi=mh$ with $m=n/\ord(\xi)$, then we may define the \emph{mask polynomial} of $A,\xi$ to be
$$P_{A,\xi}(z)=\sum_{a\in A}z^{\bra{a,h}}, \quad\text{where}\quad \xi=\frac{n}{\ord(\xi)}h.$$
Comparing it with \eqref{EqDefF} we see that
$$\EF{A}(\xi)=\sum_{a\in A}e^{2\pi i\bra{a,\xi}/n}=\sum_{a\in A}e^{2\pi i\bra{a,mh}/n}=\sum_{a\in A}e^{2\pi i\bra{a,h}/\ord(\xi)}=P_{A,\xi}(e^{2\pi i/\ord(\xi)}).$$ 
Therefore the Fourier transform $\EF{A}$ evaluated at $\xi$ is at the same time the evaluation of the mask polynomial $P_{A,\xi}(z)$ at the $\ord(\xi)$-th root of unity $z=e^{2\pi i/\ord(\xi)}$. 

For each $n\in\mathbb N$, denote the $n$-th cyclotomic polynomial by $\Phi_n(z)$. Clearly $P_{A,\xi}(z)\in \mathbb N[z]$, therefore whether $\xi\in \EZ{A}$ holds depends on the divisibility of $P_{A,\xi}(z)$ by the $\ord(\xi)$-th cyclotomic polynomial. In particular, this also means that if $L$ is a linear transformation on $\G{n}$ that satisfies 
$$\bra{Lx,Ly}=d\bra{x,y},$$
(i.e., a ``unitary matrix'' scaled by $d$) then
$$\EF{LA}(L\xi)=\sum_{a\in A}e^{2\pi i\bra{La,L\xi}/n}=\sum_{a\in A}e^{2\pi i d\bra{a,\xi}/n}.$$
If $d$ is coprime to $n$, then it is also coprime to $\ord(\xi)$, thus $e^{2\pi id/\ord(\xi)}$ is still a primitive $\ord(\xi)$-th root of unity, and $z=e^{2\pi id/\ord(\xi)}$ is a root of $P_{LA,L\xi}(z)$ if and only if $z=e^{2\pi i/\ord(\xi)}$ is a root of $P_{A,\xi}(z)$. Therefore to determine whether $A$ is tiling/spectral, we can equivalently investigate whether $LA$ is tiling/spectral.

This observation allows us to apply ``change of basis'' in subsequent analysis, as we shall see later it is often needed to map a pair of generators $h,h'$ of $\G{n}$ to the ``standard basis'' $(1,0)$ and $(0,1)$. Unfortunately if $n$ is a composite number, then $\G{n}$ is not a vector space, which means an $L$ that satisfies
\begin{equation} \label{EqL}
\bra{Lx,Ly}=d\bra{x,y}, \quad \gcd(d,n)=1,
\end{equation}
($\gcd$ stands for greatest common divisor) and at the same time maps an arbitrary generating pair $h,h'$ to the standard basis may not exist. For example, $a=(1,1), b=(1,3)$ is a generating pair of $\G{4}$, they are orthogonal to each other since $\bra{a,b}=0\pmod 4$, but there is no $L$ that can map them to standard basis and at the same time keep \eqref{EqL}. Moreover, an ``orthogonal decomposition'' need not even exist in $\G{n}$, for example, consider the subgroup generated by $(1,18)$ in $\G{25}$, it is a maximal cyclic subgroup, we may also verify that
$$(-18,1)=7\cdot(1,18) \pmod{25}.$$ 
Therefore this group is ``orthogonal'' to itself, and an exhaustive search will show that it has no orthogonal complement (i.e., no other maximal cyclic subgroup is orthogonal to it). These cases can exist in $\G{n}$ when $-1$ is a quadratic residue modulo $n$.

None of these issues will persist if we switch to the symplectic setting,  with which we will be allowed to change basis freely.  

\subsection{The Symplectic setting} 
Given $x=(x_1,x_2)$ and $y=(y_1,y_2)$, define the \emph{symplectic form} between $x,y$ to be
$$\bras{x,y}=x_1y_2-y_1x_2.$$
The symplectic form enters Fourier analysis in the way that it captures commutativity of time-frequency shifts, i.e., if $T_t:f(x)\mapsto f(x-t)$ and $M_{\xi}: f(x)\mapsto e^{2\pi i\xi\cdot x}f(x)$ are respectively the translation operator and the modulation operator on $L^2(\mathbb R)$, then
$$M_aT_bM_{a'}T_{b'}=e^{2\pi i(ab'-a'b)}M_{a'}T_{b'}M_{a}T_{b},$$
holds for any $a,b,a',b'\in\mathbb R$. The term $ab'-a'b$ in the exponent is a symplectic form.

If $h=(h_1,h_2),h'=(h_1',h_2')\in\G{n}$ satisfies $\bras{h,h'}=1$, then they are called a pair of \emph{symplectic basis} on $\G{n}$. The relation 
$$\bras{h,h'}=h_1h_2'-h_1'h_2=\det\begin{pmatrix}h_1 & h_1' \\ h_2 & h_2'\end{pmatrix}$$
indicates that in such case $h,h'$ generate $\G{n}$. The following lemma is essentially \cite[Lemma 2.3]{zhou2024}.

\begin{Lm} \label{LmBasis}
If $H,H'$ are two subgroups in $\G{n}$ with $|H|=|H'|=n$ and $\G{n}=H\oplus H'$, then for each generator $h$ of $H$, there is some generator $h'$ of $H'$ such that $h,h'$ form a symplectic basis.
\end{Lm}

Define the \emph{symplectic orthogonal set} of a given $H\subseteq\G{n}$ to be
$$H^{\perp_s}=\{x\in\mathbb Z_n\times\mathbb Z_n: \bras{x,h}=0, \forall h\in H\}.$$
The following lemma can be found in \cite[Lemma 2.2]{zhou2024}, $H$ is often called a \emph{Lagrangian} if $H=H^{\perp_s}$. 

\begin{Lm} \label{LmLag}
Let $H\subseteq\G{n}$, $H=H^{\perp_s}$ if and only if $H$ is an order $n$ subgroup.  
\end{Lm}

From the proof of \cite[Lemma 2.2]{zhou2024} one also sees that if $H$ is a subgroup, then
\begin{equation} \label{EqHHs}
|H|\cdot |H^{\perp_s}|=n^2.
\end{equation}

Now define the \emph{symplectic Fourier transform} of a function $f$ on $\G{n}$ to be
$$\widehat{f}^{sym}(\xi)=\sum_{x\in\G{n}}f(x)e^{2\pi i\bras{x,\xi}/n}, \quad \xi\in\G{n}.$$
Similarly if $f=\mathbf 1_A$, then we have
$$\F{A}(\xi)=\sum_{a\in A}e^{2\pi i\bras{a,\xi}/n}.$$
The \emph{symplectic mask polynomial} of given $A$ and given $\xi=(n/\ord(\xi))\cdot h$ is accordingly
$$P_{A,\xi}^{sym}(z)=\sum_{a\in A}z^{\bras{a,h}},$$
similarly $\F{A}$ is still the evaluation of $P_{A,\xi}^{sym}(z)$ at the $\ord(\xi)$-th root of unity $z=e^{2\pi i/\ord(\xi)}$:
\begin{equation} \label{EqMask}
\F{A}(\xi)=P_{A,\xi}^{sym}(e^{2\pi i/\ord(\xi)}).
\end{equation}

We may also define a \emph{symplectic spectral set} to be a subset $A\subseteq\G{n}$ for which there exists some $S\subseteq\G{n}$ such that $\{e^{2\pi i\bras{a,s}/n}\}_{s\in S}$ forms an orthogonal basis on $L^2(A)$ with respect to the counting measure, $S$ is then called a \emph{symplectic spectrum} of $A$ and $(A,S)$ is said to be a \emph{symplectic spectral pair}. Similar as in the Euclidean setting it is easy to see that $(A,B)$ being a tiling pair is equivalent to
\begin{equation} \label{EqDecomp}
|A|\cdot|B|=|\G{n}| \quad\text{and}\quad \Delta(\G{n}) \subseteq Z(\F{A}\cdot \F{B}),
\end{equation}
while $(A,S)$ is a symplectic spectral pair in $\G{n}$ if and only if
\begin{equation} \label{EqAnn}
|A|=|S| \quad\text{and}\quad \Delta S\subseteq\Z{A}.
\end{equation}

\begin{Lm} \label{LmSEquiv}
Let $A\subseteq\G{n}$, then $A$ is a spectral set if and only if it is also a symplectic spectral set.
\end{Lm}

\begin{proof}
It is easy to see that $(s_1,s_2)\in\Z{A}$ if and only if $(-s_2,s_1)\in\Z{A}$, therefore the symplectic spectrum and the spectrum differ only by a $90$ degree rotation, which is a bijection.
\end{proof}

A \emph{symplectic transformation} on $\G{n}$ is a linear transformation $L$ such that
\begin{equation} \label{EqLs}
\bras{Lx,Ly}=\bras{x,y}, \quad \forall x,y\in\G{n},
\end{equation}
holds. On $\G{n}$ it is just a $2\times 2$ matrix with unit determinant modulo $n$. Indeed, let $L=\begin{pmatrix}a & b \\ c & d\end{pmatrix}$ and $S=\begin{pmatrix}0 & 1 \\ -1 & 0\end{pmatrix}$, then $\bras{x,y}=\bra{x,Sy}$ implies that \eqref{EqLs} holds if and only if $L^*SL=S$ ($L^*$ is the adjoint of $L$), i.e.,
$$\begin{pmatrix}a & c \\ b & d\end{pmatrix}\begin{pmatrix}0 & 1 \\ -1 & 0\end{pmatrix}\begin{pmatrix}a & b \\ c & d\end{pmatrix}=\begin{pmatrix}0 & 1 \\ -1 & 0\end{pmatrix}.$$ 
Straightforward computation shows
$$\begin{pmatrix}a & c \\ b & d\end{pmatrix}\begin{pmatrix}0 & 1 \\ -1 & 0\end{pmatrix}\begin{pmatrix}a & b \\ c & d\end{pmatrix}=\begin{pmatrix}0 & ad-bc \\ bc-ad & 0\end{pmatrix},$$
which means $L$ is a symplectic transformation if and only if $ad-bc=1 \pmod n$.

\begin{Lm} \label{LmTrans}
If $x,y$ is a symplectic basis, then there is a symplectic transformation that maps $x$ to $(1,0)$ and $y$ to $(0,1)$. 
\end{Lm}

\begin{proof}
Suppose $x=(x_1,x_2), y=(y_1,y_2)$, then $\begin{pmatrix}x_1 & y_1 \\ x_2 & y_2\end{pmatrix}$ maps $(1,0)$ to $x$ and $(0,1)$ to $y$, hence its inverse $\begin{pmatrix}y_2 & -y_1 \\ -x_2 & x_1\end{pmatrix}$ has the desired mapping property, it is a symplectic transformation since its determinant is $x_1y_2-x_2y_1=\bras{x,y}=1$.
\end{proof}

In the context of this paper, \eqref{EqLs} can also be replaced with $\bras{Lx,Ly}=d\bras{x,y}$ for any $d$ that is coprime to $n$. Similarly $\bras{x,y}=d$ for $d$ coprime to $n$ suffices to make $x,y$ a generating pair of $\G{n}$. We choose to set $d=1$ in both cases to be consistent with the literature and for expository purpose. This choice is harmless since if $x,y$ is a generating pair, then Lemma \ref{LmBasis} shows that it is always possible to extract a symplectic basis in the respective maximal cyclic subgroups they generate. Lemma \ref{LmGen} below is \cite[Lemma 2.4]{zhou2024}.

\begin{Lm} \label{LmGen}
Let $A\subseteq\G{n}$ and $h,h'\in\G{n}$, if $h$ and $h'$ generate the same cyclic subgroup, then $h\in\Z{A}$ implies $h'\in\Z{A}$.
\end{Lm}

\begin{Lm} \label{LmZ}
Let $H$ be a subgroup of $\G{n}$, then 
\begin{enumerate}[leftmargin=*, label=(\roman*)]
\item \label{Z1}
$$\F{H}=|H|\cdot \mathbf 1_{H^{\perp_s}}.$$
\item \label{Z2} $A\subseteq\G{n}$ complements $H$ if and only if  
$$|A|\cdot |H|=n^2 \quad\text{and}\quad \Delta H^{\perp_s}\subseteq\Z{A}.$$
\end{enumerate} 
\end{Lm}

Lemma \ref{LmZ} \ref{Z1} \ref{Z2} are essentially \cite[Lemma 2.5, Lemma 2.6]{zhou2024} respectively. In particular, Lemma \ref{LmZ} \ref{Z1} can be viewed as the counterpart of the Poisson summation formula. 

From now on we use the notation $\bra{x}$ for the subgroup generated by $x$ in $\G{n}$. The next lemma is an equi-distribution phenomenon as described in \cite{shi2020} (and reference therein).
\begin{Lm}[counting] \label{LmCount}
Let $p$ be a prime number, $x\in\G{n}$ with $\ord(x)=p^m$ ($m\in\mathbb N$) and $A\subseteq\G{n}$, if $x\in\Z{A}$, then 
\begin{equation} \label{EqCount}
|A\cap\bra{px}^{\perp_s}|=|A\cap\bra{x}^{\perp_s}|\cdot p.
\end{equation}
Moreover, if 
$$V_k=\{y\in\G{n}: \bras{x,y}=k\cdot \frac{n}{p} \pmod n\},$$
then we have
\begin{equation} \label{EqEqui}
|A\cap V_0|=|A\cap V_1|=\ldots=|A\cap V_{p-1}|.
\end{equation}
\end{Lm}

\begin{proof}
Without loss of generality (by using a symplectic transformation as in Lemma \ref{LmTrans}) let us assume that $x=(0,n/p^m)$, then it is straightforward to verify that
$$\bra{x}^{\perp_s}=\{0,p^m,2p^m,\ldots,(n/p^m-1)p^m\}\times\mathbb Z_n.$$
and 
$$\bra{px}^{\perp_s}=\{0,p^{m-1},2p^{m-1},\ldots,(n/p^{m-1}-1)p^{m-1}\}\times\mathbb Z_n.$$

Setting $\omega=e^{2\pi i/p^m}$, by \eqref{EqMask} we get
\begin{equation} \label{EqCountd3}
\F{A}(x)=P_{A,x}^{sym}(\omega)=\sum_{(a_1,a_2)\in A}\omega^{a_1},
\end{equation}
then $\F{A}(x)=0$ implies that the symplectic mask polynomial is divisible by the $p^m$-th cyclotomic polynomial $1+z^{p^{m-1}}+\ldots+z^{(p-1)p^{m-1}}$, therefore we get 
$$\sum_{(a_1,a_2)\in A}\omega^{a_1}=\left(\sum_{j=0}^{p-1}\omega^{j\cdot p^{m-1}}\right) \cdot \left(\sum_{k=0}^{p^{m-1}-1}c_k\omega^k\right),$$
with $c_k\in\mathbb Z$. Expand it further we obtain
\begin{equation} \label{EqExpand}
\left(\sum_{j=0}^{p-1}\omega^{j\cdot p^{m-1}}\right) \cdot \left(\sum_{k=0}^{p^{m-1}-1}c_k\omega^k\right)=c_0\left(1+\omega^{p^{m-1}}+\ldots+\omega^{(p-1)p^{m-1}}\right)+g(\omega),
\end{equation}
where
$$g(\omega)=\left(\sum_{j=0}^{p-1}\omega^{jp^{m-1}}\right) \cdot \left(\sum_{k=1}^{p^{m-1}-1}c_k\omega_k\right),$$
collects all terms of form $c_k\omega^k$ with the exponent $k$ being a number that is not divisible by $p^{m-1}$. 

Comparing the part $c_0(1+\omega^{p^{m-1}}+\ldots+\omega^{(p-1)p^{m-1}})$ in \eqref{EqExpand} with \eqref{EqCountd3} we see that the terms $c_0,c_0\omega^{p^{m-1}}, \ldots, c_0\omega^{(p-1)p^{m-1}}$ must be obtained from elements in $A$ whose first coordinate is divisible by $p^{m-1}$, i.e., elements in $\bra{px}^{\perp_s}$, therefore we get
$$\F{A\cap\bra{xp}^{\perp_s}}(x)=c_0\left(1+\omega^{p^{m-1}}+\ldots+\omega^{(p-1)p^{m-1}}\right).$$  
Counting the number of terms then shows that 
\begin{equation} \label{EqCountP}
|A\cap\bra{px}^{\perp_s}|=c_0p.
\end{equation}
By the same reason it is also clear that 
$$c_0=|A\cap\bra{x}^{\perp_s}|.$$
Plug it back into \eqref{EqCountP} produces \eqref{EqCount}. And \eqref{EqEqui} follows from noticing that
$$\F{A\cap V_k}(x)=c_0\omega^{kp^{m-1}},$$
which means $|A\cap V_k|=c_0$ holds for each $k\in\{0,1,\ldots,p-1\}$.
\end{proof}

The following uncertainty principle \cite{donoho1989,smith1990} is well known
$$|\supp(f)|\cdot |\supp(\widehat f)|\ge n^2.$$
Now since the symplectic zero set is just a rotation of the Euclidean zero zero, this easily translates into:
\begin{equation} \label{EqUncertainty}
|\supp(f)|\cdot |\supp(\widehat f^{sym})|\ge n^2.
\end{equation}

\begin{Lm} \label{LmUncertainty}
Let $A,B$ be subsets of $\G{n}$, and $H$ a subgroup of $\G{n}$.
\begin{enumerate}[leftmargin=*, label=(\roman*)]
\item If $|A|<|H|$, then $\Delta H$ can not be completely contained in $\Z{A}$. \label{U1}
\item If $\Delta H\subseteq Z(\F{A}\cdot\F{B})$, then $|A|\ge |H|/|B|$. \label{U2}
\end{enumerate}
\end{Lm}

\begin{proof}
\ref{U1} is clear since otherwise $(A,H)$ would be a spectral pair with $|A|<|H|$.

For \ref{U2}, from Lemma \ref{LmZ} \ref{Z1} we have
$$\Z{H^{\perp_s}}=\G{n}\setminus(H^{\perp_s})^{\perp_s}=\G{n}\setminus H.$$
Therefore if $\Delta H\subseteq Z(\F{A}\cdot \F{B})$, then we get
$$\supp(\F{A}\cdot \F{B}\cdot \F{H^{\perp_s}})=\{(0,0)\}.$$
Let $f=\mathbf 1_{A+B+H^{\perp_s}}$, by \eqref{EqConv} we have
$$\widehat f=\F{A}\cdot \F{B}\cdot \F{H^{\perp_s}}.$$
By \eqref{EqUncertainty} we see that
$$|A+B+H^{\perp_s}|=|\supp(f)|\ge \frac{n^2}{|\supp(\F{A}\cdot \F{B}\cdot \F{H^{\perp_s}})|}=n^2,$$
while if $|A|<|H|/|B|$, then by \eqref{EqHHs} we would have
$$|A+B+H^{\perp_s}|\le|A|\cdot |B|\cdot |H^{\perp_s}|<|H|\cdot |H^{\perp_s}|=n^2,$$
which is a contradiction.
\end{proof}

\section{Some results in $\G{p^m}$} \label{SecPm}
Consider the partition of $\G{p^m}$ induced by the equivalence relation:
\begin{equation} \label{EqEquiv}
\sim: h\sim h'\text{ if $h,h'$ generate the same cyclic subgroup in } \G{p^m}.
\end{equation}
In this section we will use $E_{(a,b)}$ for the equivalence class that contains $(a,b)$, and call it non-trivial if $(a,b)\neq(0,0)$. Notice that $E_{(a,b)}$ is exactly the set of generators of the cyclic subgroup generated by $(a,b)$.

Given $A\subseteq\G{p^m}$, by Lemma \ref{LmGen} we see that a non-trivial equivalence class is either completely contained in $\Z{A}$ or completely disjoint with it. Moreover, for each $E_{(a,b)}$ there is a \emph{derived set} $C_{(a,b)}$ defined as
$$C_{(a,b)}=\{(0,0),(a,b),2(a,b),\ldots,(p-1)(a,b)\},$$
that satifies
\begin{equation} \label{EqDerive}
|C_{(a,b)}|=p, \quad \Delta C_{(a,b)}\subseteq E_{(a,b)}.
\end{equation}
This is clear by considering the group isomorphism $(a,b)\mapsto 1$ from $\bra{(a,b)}$ to $\mathbb Z_{\ord((a,b))}$.

\begin{Lm} \label{LmBasic}
If $A$ is a tile in $\G{p^m}$ with $|A|>1$, then $\Z{A}\cap\Delta(\G{p^m})\neq\emptyset$. Alternatively if $A$ is a spectral set in $\G{p^m}$, then $p$ divides $|A|$.
\end{Lm}

\begin{proof}
By \eqref{EqDecomp}, if $\Z{A}\cap\Delta(\G{p^m})$ is empty, then $\Delta(\G{p^m})$ is contained in $\Z{B}$, which contradicts Lemma \ref{LmUncertainty} \ref{U1} since $|A|>1$ implies $|B|=|\G{p^m}|/|A|<|\G{p^m}|$. This establishes the first statement.

By Lemma \ref{LmSEquiv}, $A$ being spectral implies it is also symplectic spectral, then we can view $\F{A}(s)$ as its symplectic mask polynomial evaluated at $z=e^{2\pi i/\ord(s)}$. As it vanishes the polynomial must be divisible by $\Phi_{\ord(s)}(z)$, consequently $|A|=\F{A}(0)$ is divisible by $\Phi_{\ord(s)}(1)=p$ (since $\ord(s)$ is a power of $p$). This establishes the second statement.
\end{proof}

\begin{Lm} \label{LmTSp}
If $A$ is a tile in $\G{p^m}$ with $|A|=p$, then $A$ is spectral.
\end{Lm}

\begin{proof}
By Lemma \ref{LmBasic}, $\Z{A}$ is not empty, thus by Lemma \ref{LmGen} there is some non-trivial equivalence class $E_{(a,b)}$ that is completely contained in $\Z{A}$, and therefore by \eqref{EqDerive} we see that the derived set $C_{(a,b)}$ is a symplectic spectrum of $A$, which by Lemma \ref{LmSEquiv} further implies that $A$ is spectral.
\end{proof}

The cyclic case of the following dilation result is established in \cite[Theorem 1]{tijdeman1995}, the non-cyclic case is later proved in \cite[Proposition 3]{sands2000}.
\begin{Prop} \label{PropScale}
Let $G$ be an additive finite Abelian group, and $(A,B)$ a tiling pair in $G$. If $q\in\mathbb N$ is co-prime to $|A|$, then $(qA,B)$ is still a tiling pair in $G$.
\end{Prop}

\begin{Lm} \label{LmDiff} 
If $E$ is an equivalence class under $\sim$ as defined in \eqref{EqEquiv}, and $(A,B)$ is a tiling pair in $\G{p^m}$ with $\Delta A\cap E\neq\emptyset$, then $\Delta B\cap E=\emptyset$.
\end{Lm}

\begin{proof}
Suppose that $E$ generates a cyclic subgroup of size $p^t$, then we can view such a subgroup as the multiplicative group mod $p^t$, consequently $E\to qE$ is a bijection of $E$ itself for any $q$ that is co-prime to $p$. If on the contrary there are $a\in \Delta A\cap E$ and $b\in \Delta B\cap E$, then we can find some $q$ so that $qa\mapsto b$. By Proposition \ref{PropScale} we see that $(qA,B)$ is still a tiling pair but then we would have
$$b\in \Delta(qA)\cap\Delta B,$$
which contradicts \eqref{EqDiff}.
\end{proof}

\begin{Lm} \label{LmTSq}
If $A$ is a tile in $\G{p^m}$ with $|A|=p^{2m-1}$, then $A$ is spectral.
\end{Lm}

\begin{proof}
By Lemma \ref{LmDiff} there should be a non-trivial equivalence class that is disjoint with $\Delta A$, otherwise $B$ has to be trivial. Without loss of generality (by using a symplectic transformation as in Lemma \ref{LmTrans}) we may likewise assume that it is $E_{(p^{m-t},0)}$ for some $t$, then we get
$$\Delta A\subseteq (\G{p^m})\setminus (E_{(p^{m-t},0)}\cup\{(0,0)\}).$$
It is easy to verify that if we take
$$S=\mathbb Z_{p^m}\times D,$$
where
$$D=\{d_1\ldots d_{t-1}0d_{t+1}\ldots d_m: d_k\in\mathbb Z_p, k\neq t\},$$
contains all elements in $\mathbb Z_{p^m}$ whose base-$p$ expansion is $0$ at $t$-th digit, then we have
\begin{equation} \label{EqTSq}
\Z{S}=\Z{\mathbb Z_{p^m}\times\{0\}}\cup \Z{\{0\}\times D}=(\G{p^m})\setminus \left(E_{(p^{m-t},0)}\cup\{(0,0)\}\right).
\end{equation}
Indeed, in the right-hand side, $\mathbb Z_{p^m}\times\{0\}$ is a maximal cyclic subgroup, therefore by Lemma \ref{LmZ} \ref{Z1} we obtain
$$\Z{\mathbb Z_{p^m}\times\{0\}}=(\G{p^m})\setminus \left(\mathbb Z_{p^m}\times\{0\}\right),$$
while plugging $(a,0)\in\mathbb Z_{p^m}\times\{0\}$ into $\F{\{0\}\times D}$ yields
$$\F{\{0\}\times D}(a,0)=\prod_{\substack{k=1 \\ k\neq t}}^m\Phi_{p^k}(e^{2\pi ia/p^m}),$$
which is $0$ if $(a,0)\notin E_{(p^{m-t},0)}\cup\{(0,0)\}$ and non-zero otherwise. Therefore \eqref{EqTSq} holds and thus $\Delta A\subseteq \Z{S}$, it then follows from \eqref{EqAnn} that $(A,S)$ is a symplectic spectral pair, which by Lemma \ref{LmSEquiv} further implies that $A$ is spectral.
\end{proof}

\begin{Lm} \label{LmSTq}
If $A$ is a spectral set in $\G{p^m}$ with $p^{2m}>|A|\ge p^{2m-1}$, then $A$ must be a tile with $|A|=p^{2m-1}$.
\end{Lm}

\begin{proof}
By Lemma \ref{LmSEquiv} $A$ is also symplectic spectral, let $S$ be a symplectic spectrum of $A$, first we notice that there should be some non-trivial equivalence class $E_{(a,b)}$ that is disjoint with $\Delta A$ (otherwise we would have $\Delta(\G{p^m})\subseteq \Z{S}$ which contradicts Lemma \ref{LmUncertainty} \ref{U1}), but then for the derived set $C_{(a,b)}$ we would have
$$\Delta A\cap \Delta C_{(a,b)}\subseteq \Delta A\cap E_{(a,b)}=\emptyset,$$
which by Lemma \ref{LmDelta} shows that $(A, C_{(a,b)})$ is a tiling pair and $|A|=p^{2m-1}$ (since $|C_{(a,b)}|=p$).
\end{proof}

\begin{Lm} \label{LmSTp}
If $A\subset\G{p^m}$ is a spectral set of size $p$, then it is also a tile.
\end{Lm}

\begin{proof}
By Lemma \ref{LmSEquiv}, $A$ is also symplectic spectral set, thus $\Z{A}$ must be non-empty by Lemma \ref{LmBasic}. By Lemma \ref{LmGen} this means that at least one of the non-trivial equivalence classes must be completely contained in it, by using a symplectic transformation (as in Lemma \ref{LmTrans}) we may without loss of generality assume it is $E_{p^t,0}$ for some $t$.

Let $S$ be the same set as in the proof of Lemma \ref{LmTSq}, then by \eqref{EqTSq} we get
$$\Delta(\G{p^2})\subseteq Z(\F{A}\cdot\F{S}),$$
which by \eqref{EqDecomp} means that $(A,S)$ is a tiling pair.
\end{proof}

\section{Structures in $\G{p^2}$} \label{SecP2}
\subsection{Notations}
Now let us recall the structure of $\G{p^2}$: there is a unique subgroup formed by all order $p$ elements, we denote it by $K$. Counting the number of generators, it is easy to see that there are $p+1$ number of proper and non-trivial subgroups in $K$, each of them is cyclic and they mutually intersect trivially. Similar arguments indicate that there are $p^2+p$ maximal cyclic subgroups in $\G{p^2}$, they can be divided into $p+1$ classes depending on which order $p$ subgroup do they contain, and there are precisely $p$ number of maximal cyclic subgroups in each class.

Since it is important to keep track of these classes we shall introduce the following projective notation: The order $p$ cyclic subgroup generated by $(ap,bp)$ will be denoted by $K_{b/a}$ (with the convention $b/a=\infty$ if $a=0$) so that $K_0, K_1,\ldots, K_{p-1}, K_{\infty}$ is the list of all order $p$ cyclic subgroups. For each $j\in\mathbb Z_p$ and each $k\in\{0,1,\ldots,p-1,\infty\}$ set
$$h_{j,k}=\begin{cases}(1,jp+k) & k\neq\infty, \\ (jp,1) & k=\infty.\end{cases}.$$
The maximal cyclic subgroup generated by $h_{j,k}$ will be denoted by $H_{j,k}$, so that the second subscript $k$ in $H_{j,k}$ actually indicate $H_{j,k}$ contains $K_k$, i.e.,
$$\bigcap_{j\in\mathbb Z_p}H_{j,k}=K_k, \quad K_k^{\perp_s}=(\bigcup_{j\in\mathbb Z_p}H_{j,k})\bigcup K.$$

For convenience, we slightly adjust the notation for equivalence classes used in the last section: from now on $E_{j,k}$ is the set of all generators in $H_{j,k}$, and 
$$C_{j,k}=\{0,h_{j,k}, 2h_{j,k},\ldots,(p-1)h_{j,k}\},$$ 
is its derived set. The following two pictures may help to build intuitions behind these notations.

\begin{figure}[H]
\begin{center}
\begin{tikzpicture}[scale=0.9]
%H_0
\draw [fill=gray!30] (0,0) rectangle (4,1);
\draw (0,0) grid (4,4);
\node at (2.5, -0.75) {{\footnotesize{$H_{0,0}$}}};
%H_2
\draw (5,0) grid (9,4);
\draw [fill=gray!30] (5,0) rectangle (6,1);
\draw [fill=gray!30] (6,2) rectangle (7,3);
\draw [fill=gray!30] (7,0) rectangle (8,1);
\draw [fill=gray!30] (8,2) rectangle (9,3);
\node at (7.5, -0.75) {{\footnotesize{$H_{1,0}$}}};
%H_1
\draw (10,0) grid (14,4);
\draw [fill=gray!30] (10,0) rectangle (11,1);
\draw [fill=gray!30] (11,1) rectangle (12,2);
\draw [fill=gray!30] (12,2) rectangle (13,3);
\draw [fill=gray!30] (13,3) rectangle (14,4);
\node at (12.5, -0.75) {{\footnotesize{$H_{0,1}$}}};
%H_3
\draw (15,0) grid (19,4);
\draw [fill=gray!30] (15,0) rectangle (16,1);
\draw [fill=gray!30] (16,3) rectangle (17,4);
\draw [fill=gray!30] (17,2) rectangle (18,3);
\draw [fill=gray!30] (18,1) rectangle (19,2);
\node at (17.5, -0.75) {{\footnotesize{$H_{1,1}$}}};
%H_\infty
\draw [fill=gray!30] (20,0) rectangle (21,4);
\draw (20,0) grid (24,4);
\node at (22.5, -0.75) {{\footnotesize{$H_{0,\infty}$}}};
%H_1/2
\draw (25,0) grid (29,4);
\draw [fill=gray!30] (25,0) rectangle (26,1);
\draw [fill=gray!30] (27,1) rectangle (28,2);
\draw [fill=gray!30] (25,2) rectangle (26,3);
\draw [fill=gray!30] (27,3) rectangle (28,4);
\node at (27.5, -0.75) {{\footnotesize{$H_{1,\infty}$}}};
%K
\draw (30,0) grid (34,4);
\draw [fill=gray!30] (30,0) rectangle (31,1);
\draw [fill=gray!30] (32,0) rectangle (33,1);
\draw [fill=gray!30] (30,2) rectangle (31,3);
\draw [fill=gray!30] (32,2) rectangle (33,3);
\node at (32, -0.7) {{\footnotesize{$K$}}};
\end{tikzpicture}
\caption{Order $4$ subgroups in $\mathbb Z_4\times\mathbb Z_4$} \label{FigCosets}
\end{center}
\end{figure}

\begin{figure}[H]
\begin{center}
\begin{tikzpicture}[scale=0.9]
%Ht
\draw [fill=gray!30] (0,0) rectangle (1,1);
\draw [fill=gray!30] (2,0) rectangle (3,1);
\draw (0,0) grid (4,4);
\node at (2.5, -0.75) {{\footnotesize{$K_0$}}};
%U_t
\draw [fill=gray!30] (6,0) rectangle (7,1);
\draw [fill=gray!30] (8,0) rectangle (9,1);
\draw [fill=gray!30] (6,2) rectangle (7,3);
\draw [fill=gray!30] (8,2) rectangle (9,3);
\draw (5,0) grid (9,4);
\node at (7.5, -1) {{\footnotesize{$\bigcup\limits_{j\in\mathbb Z_p}E_{j,0}$}}};
%H_d
\draw (10,0) grid (14,4);
\draw [fill=gray!30] (10,0) rectangle (11,1);
\draw [fill=gray!30] (12,2) rectangle (13,3);
\node at (12.5, -0.75) {{\footnotesize{$K_1$}}};
%U_d
\draw (15,0) grid (19,4);
\draw [fill=gray!30] (16,1) rectangle (17,2);
\draw [fill=gray!30] (18,3) rectangle (19,4);
\draw [fill=gray!30] (16,3) rectangle (17,4);
\draw [fill=gray!30] (18,1) rectangle (19,2);
\node at (17.5, -1) {{\footnotesize{$\bigcup\limits_{j\in\mathbb Z_p}E_{j,1}$}}};
%H_m
\draw (20,0) grid (24,4);
\draw [fill=gray!30] (20,0) rectangle (21,1);
\draw [fill=gray!30] (20,2) rectangle (21,3);
\node at (22.5, -0.75) {{\footnotesize{$K_{\infty}$}}};
%U_m
\draw [fill=gray!30] (25,1) rectangle (26,2);
\draw [fill=gray!30] (25,3) rectangle (26,4);
\draw [fill=gray!30] (27,1) rectangle (28,2);
\draw [fill=gray!30] (27,3) rectangle (28,4);
\draw (25,0) grid (29,4);
\node at (27.5, -1) {{\footnotesize{$\bigcup\limits_{j\in\mathbb Z_p}E_{j,\infty}$}}};
\end{tikzpicture}
\caption{$K_k$ and $ E_{j,k}$ in $\mathbb Z_4\times\mathbb Z_4$} \label{FigGen}
\end{center}
\end{figure}

The equivalence classes defined in the last section can be arranged into a tree, with each node being an equivalence class, and paths given by inclusion of subgroups. The root node is the equivalence class for $(0,0)$, it has $p+1$ children that form the first layer (since there are $p+1$ order $p$ subgroups). Each node in the first layer has $p$ children (since every subgroup of order $p$ is contained in $p$ number of order $p^2$ subgroups). Below is the example for $\mathbb Z_4\times \mathbb Z_4$: 
\begin{figure}[H]
\begin{center}
\begin{tikzpicture}
[
    level 1/.style={sibling distance=50mm},
    level 2/.style={sibling distance=25mm},
]

  \node {\scriptsize{$\{(0,0)\}$}}
    child {node {\scriptsize{$\Delta K_0$}}
      child {node {\scriptsize{$E_{0,0}$}}}
      child {node {\scriptsize{$E_{1,0}$}}}
    }
    child {node {\scriptsize{$\Delta K_1$}}
      child {node {\scriptsize{$E_{0,1}$}}}
      child {node {\scriptsize{$E_{1,1}$}}}
    }
    child {node {\scriptsize{$\Delta K_{\infty}$}}
      child {node {\scriptsize{$E_{0,\infty}$}}}
      child {node {\scriptsize{$E_{1,\infty}$}}}
    };
\end{tikzpicture}
\caption{The tree of equivalence classes in $\mathbb Z_4\times\mathbb Z_4$}
\end{center}
\end{figure}

Every path starting from the root will uniquely corresponds to a subgroup. This tree is similar to the homogenous tree introduced in \cite{fan2016, fan2019}, and it helps to keep track of different cases in the last section, in which we will discuss whether the difference set or the zero set contains or excludes branches and layers in this tree.

\subsection{Mutual Annihilation}
In this subsection we describe a special phenomenon in $\G{p^2}$, the motivating example is when $A=\{0,\ldots,p-1\}\times\{0,\ldots,p-1\}$, then $A$ is a not only a tiling complement of $K$, but also its symplectic spectrum. This is because $\Z{A}$ is disjoint with $\Delta A$, we show that this disjointness property is owned by all tiling complements of $K$.

\begin{Lm} \label{LmSelf}
If $A\subset\G{p^2}$ is a tiling complement of $K$, then 
$$A\cap\Z{A}=\emptyset.$$
\end{Lm}

\begin{proof}
Without loss of generality (by shifting, which does change $\Z{A}$ we may assume that $(0,0)\in A$. And clearly $(0,0)$ is not in $\Z{A}$.

Suppose $a\in A\setminus\{(0,0)\}$, then $\ord(a)=p^2$ since $A$ is a tiling complement of $K$, then by Lemma \ref{LmLag} we have
$$\bra{a}=\bra{a}^{\perp_s}.$$
If $a\in \Z{A}$, then by Lemma \ref{LmCount} we will have
$$|A\cap\bra{pa}^{\perp_s}|=p|A\cap\bra{a}^{\perp_s}|.$$
Observe that $|\bra{pa}|=p$, thus by \eqref{EqHHs} we have $|\bra{pa}^{\perp_s}|=|\G{p^2}|/|\bra{pa}|=p^3$,  and $K$ is a subset of $\bra{pa}^{\perp_s}$. Since $|K|=p^2$, and $A$ is a tiling complement of $K$, we obtain
$$|A\cap\bra{pa}^{\perp_s}|\le\frac{|\bra{pa}^{\perp_s}|}{|K|}=p.$$
Combing these three equations above we get
$$p\ge|A\cap\bra{pa}^{\perp_s}|=p|A\cap\bra{a}^{\perp_s}|=p|A\cap\bra{a}|,$$
i.e.,
$$|A\cap\bra{a}|\le1,$$
which is a contradiction since $|A\cap\bra{a}|\ge 2$ must hold because $A\cap\bra{a}$ contains at least the identity and $a$ itself.
\end{proof}
 
\begin{Lm} \label{LmKC}
If $A$ is a tiling complement of $K$ in $\G{p^2}$, then
\begin{equation} \label{EqKC}
\Delta A\cap\Z{A}=\emptyset.
\end{equation}
and $(A,B)$ is a symplectic spectral pair for any tiling complement $B$ of $A$, while $(A,S)$ is a tiling pair for any symplectic spectrum $S$ of $A$.
\end{Lm}

\begin{proof}
Assume the contrary that there exist distinct $a,a'\in A$ such that $a-a'\in \Delta A\cap\Z{A}$. Consider now the set
$$A'=A-a',$$
which contains $a-a'$, and thus
\begin{equation} \label{EqDTemp}
a-a'\in A'\cap\Z{A}.
\end{equation}
Clearly $|A'|=|A|$, $\Delta A'=\Delta A$ hold, hence $A'$ is still a tiling complement of $K$, consequently by Lemma \ref{LmSelf} we shall have 
$$A' \cap\Z{A'}=\emptyset.$$
On the other hand, it is also easy to see that $\Z{A'}=\Z{A}$, therefore we shall have
$$A' \cap\Z{A}=A' \cap\Z{A'}=\emptyset,$$
which contradicts \eqref{EqDTemp}. This establishes \eqref{EqKC}.

Now if $(A,B)$ is a tiling pair, then since $\Delta A\cap\Z{A}=\emptyset$, by \eqref{EqDecomp} we must have $\Delta A\subseteq\Z{B}$, which indicates that $(A,B)$ is also a symplectic spectral pair. Conversely if $(A,S)$ is a symplectic spectral pair, then $|A|=|S|=p^2$, and $\Delta S\subseteq \Z{A}$, consequently we get
$$\Delta A\cap\Delta S\subseteq\Delta A\cap \Z{A}=\emptyset,$$
which by Lemma \ref{LmDelta} indicates that $(A,S)$ is also a tiling pair.
\end{proof}

\subsection{Some important cases}
\begin{Lm} \label{LmP2}
Let $A$ be a subset of $\G{p^2}$.
\begin{enumerate}[label=(\roman*), leftmargin=*]
\item \label{P2TS} If $|A|=p^2$ and there are distinct $k,k'\in\{0, 1, \ldots, p-1, \infty\}$ such that
$$\Delta K_{k}\subseteq \Z{A}, \quad \left(\bigcup_{j\in\mathbb Z_p}E_{j,k'}\right)\subseteq \Z{A},$$
then $A$ is a symplectic spectral set in $\G{p^2}$.

\item \label{P2ST} If $|A|\ge p^2$ and there are distinct $k,k'\in\{0, 1, \ldots, p-1, \infty\}$ such that
$$\Delta K_k\cap \Delta A=\left(\bigcup_{j\in\mathbb Z_p}E_{j,k'}\right)\cap \Delta A=\emptyset,$$
then $A$ is a tiling set in $\G{p^2}$ with $|A|=p^2$.
\end{enumerate}
\end{Lm}

\begin{proof}
Without loss of generality (by using a proper symplectic transformation as in Lemma \ref{LmTrans}) we may assume $k=0$, $k'=\infty$. 

Every element of $\Delta(K_0\oplus C_{0,\infty})$ can be written as $a+b$ with 
$$a\in\{(0,0)\}\cup\Delta K_0, \quad b\in\{(0,0)\}\cup \Delta  C_{0,\infty},$$
and $a,b$ can not be $0$ simultaneously (since $(0,0)\notin\Delta(K_0\oplus C_{0,\infty})$). It is also clear that
$$\{(0,0)\}\cup\Delta K_0=K_0, \quad \{(0,0)\}\cup \Delta  C_{0,\infty}\subseteq \{0\}\cup E_{0,\infty}.$$

If $b=0$, then $a$ must not be $0$, and $a+b=a$ is in $\Delta K_0$. Else if $b\neq0$, then $a+b\in K_0\oplus E_{0,\infty}$, but
$$K_0\oplus E_{0,\infty}=K_0\oplus (H_{0,\infty}\setminus K_{\infty})=(K_0\oplus H_{0,\infty})\setminus (K_0\oplus K_{\infty})=K_0^{\perp_s}\setminus K=\bigcup_{j\in\mathbb Z_p}E_{j,\infty},$$ 
therefore in this case $a+b$ is an element of $\bigcup_{j\in\mathbb Z_p}E_{j,\infty}$. Together we can conclude that
\begin{equation} \label{EqP2}
\Delta(K_0\oplus C_{0,\infty})\subseteq \Delta K_0\bigcup\left(\bigcup_{j\in\mathbb Z_p}E_{j,\infty}\right),
\end{equation}
always holds.

The assumption of \ref{P2TS} combined with $\eqref{EqP2}$ indicates that $\Delta(K_0\oplus C_{0,\infty})\subseteq\Z{A}$, thus $(A,K_0\oplus C_{0,\infty})$ is a symplectic spectral pair by \eqref{EqAnn}, this establishes \ref{P2TS}.

The assumption of \ref{P2ST} combined with $\eqref{EqP2}$ indicates that $\Delta(K_0\oplus C_{0,\infty})\cap \Delta A=\emptyset$. Also we have
$$|A\oplus (K_0\oplus C_{0,\infty})|=|A|\cdot |K_0\oplus C_{0,\infty}|\ge p^2\cdot p^2=|\mathbb Z_n\times\mathbb Z_n|,$$
therefore by Lemma \ref{LmDelta} we can conclude that $(A,K_0\oplus C_{0,\infty})$ is a tiling pair and $|A|$ must be $p^2$. 
\end{proof}

\begin{Lm} \label{LmPL}
Let $A\subset\G{p^2}$ with $|A|=p^2$, if there exist distinct $k_0,k_1\in\{0, 1, \ldots, p-1, \infty\}$ as well as some $j_0\in\mathbb Z_p$ such that 
$$\Delta H_{j_0,k_0}\subseteq \Z{A}, \quad \left(\bigcup_{j\in\mathbb Z_p}E_{j,k_1}\right)\subseteq \Z{A}.$$ 
then 
$$\left(\bigcup_{j\in\mathbb Z_p}E_{j,k_0}\right)\subseteq \Z{A}.$$
\end{Lm}

\begin{proof}
Without loss of generality (by using a proper symplectic transformation as in Lemma \ref{LmTrans}) we may assume $j_0=k_0=0$ and $k_1=\infty$, i.e.,
$$\Delta H_{0,0}\subseteq \Z{A}, \quad \left(\bigcup_{j\in\mathbb Z_p}E_{j,\infty}\right)\subseteq \Z{A}.$$
By shifting (which changes neither $\Z{A}$ nor $|A|$) we may also assume that $A$ contains $(0,0)$. Let $P_1$ and $P_2$ be coordinate projections so that if $a=(a_1,a_2)$, then
$$P_1a=(a_1,0), \quad P_2a=(0,a_2).$$
Take an arbitrary $(s,0)\in\Delta H_{0,0}$, then 
$$\F{A}(s,0)=\sum_{(a_1,a_2)\in A}e^{-2\pi ia_2s/p^2}.$$
is the mask polynomial $\sum_{(a_1,a_2)\in A}z^{a_2}$ evaluated at $z=e^{-2\pi is/p^2}$. Therefore $\Delta H_{0,0}\subseteq \Z{A}$ implies that the mask polynomial has to be divisible by $\Phi_p(z)\Phi_{p^2}(z)$, and thus the multiset 
$$P_2A=\{(0,a_2): (a_1,a_2)\in A\},$$ 
must equal the set $\{0\}\times\mathbb Z_{p^2}$ (hence $P_2A$ is actually a usual set in this case). Now for each $m=0,1,\ldots,p-1$, take 
$$A_m=\{(a,b)\in A: b\equiv m \bmod p\},$$
then $A_0,\ldots, A_{p-1}$ form a partition of $A$, and $P_2A_m$ equals the set $\{0\}\times\{m+cp: c=0,1,\ldots,p-1\}$.

For each $j=0,1,\ldots,p-1$, plug $(jp,1)\in E_{j,\infty}$ into $\widehat{\mathbf 1}_A^{sym}$ we obtain
\begin{equation} \label{EqUV}
0=\F{A}(jp,1)=\sum_{m\in\mathbb Z_p}\sum_{(a,b)\in A_m}e^{2\pi i(a-pjm)/p^{2}}=u\cdot w^{(j)},
\end{equation}
where the $m$-th coordinates of $u,w^{(j)}\in\mathbb C^p$ are respectively
\begin{equation} \label{EqU}
u_m=\sum_{(a,b)\in A_m}e^{2\pi ia/p^2}, \quad w_m^{(j)}=e^{2\pi jm/p}.
\end{equation}
Notice that $\{w^{(j)}\}_{j\in\mathbb Z_p}$ is the orthogonal Fourier basis in $\mathbb C^p$, thus \eqref{EqUV} vanishes for all $j$ implies that $u=0$, i.e., each $u_m$ is $0$. On the other hand, each $u_m$ is the mask polynomial $\sum_{(a,b)\in A_m}z^a$ evaluated at $z=e^{2\pi i/p^2}$, thus its being $0$ implies that the polynomial is divisible by $\Phi_{p^2}(z)$, therefore $P_1A_{m}$ must take the form:
\begin{equation} \label{EqD}
P_1A_{m}=\{a_m+cp:c=0,1,\ldots,p-1\}\times\{0\},
\end{equation}
for some $a_m\in\{0,1,\ldots,p-1\}$. 

Finally, for each $j=0,1,\ldots,p-1$, taking $(1, jp)\in E_{j,0}$ we get
$$\F{A}(1,jp)=\sum_{m\in\mathbb Z_p}\sum_{c\in\mathbb Z_p}e^{2\pi i(a_mjp-cp-m)/p^2}=\sum_{m\in\mathbb Z_p}e^{2\pi i(jpa_m-m)/p^2}\Phi_p(e^{2\pi i/p})=0,$$
i.e., $\left(\bigcup\limits_{j\in\mathbb Z_p}E_{j,0}\right)\subseteq \Z{A}$, which is the desired result.
\end{proof}

\section{Main results} \label{SecMain}
\begin{proof}[\textbf{proof of Theorem \ref{ThmMain}}]
\emph{(i) The tiling to spectral part:}

Let $(A,B)$ be a tiling pair in $\G{p^2}$, we aim to produce a symplectic spectrum explicitly for $A$, then by Lemma \ref{LmSEquiv} this will indicate that $A$ is spectral. Since shifting changes neither $\Delta A, \Z{A}$ nor $\Delta B, \Z{B}$, let us assume that both $A$ and $B$ contain the identity element. 

The cases that $|A|=1$ and $|A|=p^4$ are trivial, the cases that $|A|=p$ and $|A|=p^3$ follow from Lemma \ref{LmTSp} and Lemma \ref{LmTSq} respectively. Hence let us focus on the case that $|A|=p^2$: 

If $\Delta K\subseteq \Z{A}$, then obviously $(A,K)$ is a symplectic spectral pair. 

Alternatively if $\Delta K\cap \Z{A}=\emptyset$, then we will have $\Delta K\subseteq \Z{B}$ by \eqref{EqDecomp}, which means that $B$ is a tiling complement of $K$ in $\G{p^2}$, then $(A, B)$ is a symplectic spectral pair by Lemma \ref{LmKC}.

Now we are left with the case that $\Delta K\cap \Z{A}\neq\emptyset$, but $\Delta K$ is not completely contained in $\Z{A}$ either. Recall that
$$\Delta K=\bigcup_{k\in\{0, 1, \ldots, p-1, \infty\}}\Delta K_k,$$
and by Lemma \ref{LmGen} each $\Delta K_k$ is either completely in $\Z{A}$ or completely disjoint with $\Z{A}$, we can thus partition the set $\{0, 1, \ldots, p-1, \infty\}$ into two parts $M$ and $M'$, so that
$$\tilde M=\{m:\Delta K_m\subseteq \Z{A}\}, \quad  \tilde M'=\{m':\Delta K_{m'}\cap \Z{A}=\emptyset\},$$
and they are both non-empty. Now there are three further subcases:

If there is some $m\in \tilde M$ and some $j\in\mathbb Z_p$ such that $E_{j,m}\cap \Z{A}\neq\emptyset$, then $\Delta H_{j,k}\subseteq \Z{A}$ by Lemma \ref{LmGen}, and thus $(A,H_{j,m})$ is a symplectic spectral pair. 

Alternatively if there is some $m'\in \tilde M'$ so that $ E_{j',m'}\cap Z(\widehat{\mathbf 1}_A^{sym})\neq\emptyset$ for all $j'\in\mathbb Z_p$, then $A$ is symplectic spectral by Lemma \ref{LmP2} \ref{P2TS}. 

Hence the final subcase is that for each $m\in \tilde M$ we have
$$\left(\bigcup_{j\in\mathbb Z_p}E_{j,m}\right)\cap \Z{A}=\emptyset,$$
while for each $m'\in \tilde M'$ there is at least one $j_{m'}\in\mathbb Z_p$ so that
$$ E_{j_{m'},m'}\cap \Z{A}=\emptyset.$$ 
By \eqref{EqDecomp} and Lemma \ref{LmGen}, we shall have 
$$\left(\bigcup_{j\in\mathbb Z_p}E_{j,m}\right)\cup\Delta H_{j_{m'},m'}\subseteq \Z{B},$$
which by Lemma \ref{LmPL} further indicates that $\left(\bigcup\limits_{j\in\mathbb Z_p}E_{j,m'}\right)\subseteq \Z{B}$. Applying this argument on each $m'\in \tilde M'$ we get
\begin{equation} \label{EqPL2}
\left(\bigcup_{m'\in M'}\Delta K_{m'}\right)\bigcup\left(\bigcup_{m'\in M'}\bigcup\limits_{j\in\mathbb Z_p}E_{j,m'}\right)=(\G{p^2})\setminus \bigcup_{m\in M}K_m\subseteq \Z{B}.
\end{equation}
which means 
$$|\supp(\F{B})|<|\bigcup_{m\in M}K_m|<|K|<p^2,$$
and thus contradicts \eqref{EqUncertainty}. These exhaust all possibilities of the tiling to spectral cases. \\

\emph{(ii) The spectral to tiling part, with set sizes $1,p,p^2,p^3,p^4$:} 

Suppose that $A$ is spectral, by Lemma \ref{LmSEquiv} this means that $A$ is also symplectic spectral. Let $S\subseteq\G{p^2}$ be a symplectic spectrum of $A$. We aim to produce a tiling complement of $A$ in $\G{p^2}$. Since shifting changes neither $\Delta A, \Z{A}$ nor $\Delta S, \Z{S}$, let us assume that both $A$ and $S$ contain the identity element. 

The cases that $|A|=1$ and $|A|=p^4$ are trivial, the cases that $|A|=p^3$ and $|A|=p$ are established in Lemma \ref{LmSTq} and \ref{LmSTp} respectively, therefore let us again focus on the case that $|A|=p^2$:

\textbf{(1)} If there is some $B$ with $\Delta A\cap\Delta B=\emptyset$, and $B$ is an order $p^2$ subgroup or a tiling complement of some subgroup, then $|B|=p^2$, and Lemma \ref{LmDelta} indicates that $(A,B)$ is a tiling pair.

\textbf{(2)} If there is no such $B$ as in (1), then $\Delta A$ shall intersect each order $p^2$ subgroup non-trivially, and at the same time not completely contained in any order $p^2$ subgroup (otherwise it equals that subgroup and we are back in the last case). Then we shall look at how $\Delta A$ intersects $\Delta K$ and how it intersects each equivalence classes. Let us partition the set $\{0,1,\ldots,p-1,\infty\}$ into $M$ and $M'$, so that
$$M=\{m: \Delta A\cap\Delta K_m\neq \emptyset\}, \quad M'=\{m': \Delta A\cap\Delta K_{m'}=\emptyset\}.$$
Notice also that, because $\Delta A$ needs to intersect each maximal cyclic subgroup $H_{j,m'}$ non-trivially, for each $m'\in M'$ (if $M'$ is non-empty) we must also have
\begin{equation} \label{EqSTP2M2}
\Delta A\cap E_{j,m'}\neq\emptyset,\quad \forall j\in\mathbb Z_p.
\end{equation}
Let us discuss the size of $M'$:

\textbf{(2.1)} If $|M'|\ge 1$, then we further look at the size of $M$: Clearly $|M|\ge 1$ must also hold, since at the beginning of case (2) we have assumed that $\Delta A$ must intersect the subgroup $K$ non-trivially, and recall that
$$K=\{(0,0)\}\cup\Delta K_0\cup\ldots\cup\Delta K_{p-1}\cup \Delta K_{\infty},$$
hence $\Delta A$ must intersect at least one of  $\Delta K_0,\ldots,\Delta K_{p-1},\Delta K_{\infty}$.

Next since $|M|+|M'|=p+1$, we shall also have 
$$|M|\le  p+1-|M'|\le p.$$
Moreover, if there is some $m\in M$ such that
\begin{equation} \label{EqSTP2M3}
\left(\bigcup_{j\in\mathbb Z_p}E_{j,m}\right)\cap\Delta A=\emptyset,
\end{equation}
then Lemma \ref{LmP2} \ref{P2ST} (with $m$ playing the role of $k'$ and any $m'\in M'$ playing the role of $k$) already shows that $A$ is a tile. 

Hence let us further handle the case that for each $m\in M$ there is some $j_m\in\mathbb Z_p$ such that
\begin{equation} \label{EqSTP2M3}
\Delta A\cap E_{j_m,m}\neq\emptyset.
\end{equation}
\textbf{(2.1.1)} If $|M|=1$, say $M=\{0\}$ and $j_0=0$ without loss of generality (by using a symplectic transformation as in Lemma \ref{LmTrans}), then \eqref{EqSTP2M2} means each component in the right-hand side of
 $$(\G{p^2})\setminus K_0^{\perp_s}=\bigcup_{m'\in M'}\bigcup_{j\in\mathbb Z_p}E_{j,m'},$$
intersects $\Delta A$ non-trivially. And \eqref{EqSTP2M3} means each component in the right-hand side of
$$\Delta H_{0,0}=E_{0,0}\cup \Delta K_0,$$
also intersects $\Delta A$ non-trivially.  By \eqref{EqAnn} we have $\Delta A\subseteq \Z{S}$, while by Lemma \ref{LmGen} this means all the equivalence classes that intersects $\Delta A$ non-trivially are annihilated by $\F{S}$, i.e.,
$$\left((\G{p^2})\setminus K_0^{\perp_s}\right)\bigcup \Delta H_{0,0}=(\G{p^2})\setminus\left(K_0^{\perp_s}\setminus\Delta H_{0,0}\right)\subseteq \Z{S}.$$
On the other hand, it is also easy to verify that
$$K_0^{\perp_s}\setminus H_{0,0}=\Z{C_{0,0}},$$
hence together we will have 
\begin{equation} \label{EqSTP2}
\Delta(\G{p^2})\subseteq Z(\F{S}\cdot \F{C_{0,0}}).
\end{equation}
Now since $C_{0,0}=\{(0,0),(1,0),\ldots,(p-1,0)\}$ has size $p$, by Lemma \ref{LmUncertainty} \ref{U2} this indicates that 
$$|S|\ge\frac{|\G{p^2}|}{|C_{0,0}|}=p^3,$$
which is a contradiction to the assumption that $p^2=|A|=|S|$. 

\textbf{(2.1.2)} If $|M|\ge 2$, then by \eqref{EqSTP2M3} we shall have distinct $k_0,k_1,k_2$ such that $k_0,k_2\in M$ and $k_1\in M'$ with
$$\Delta H_{j_{k_0},k_0}\bigcup \Delta K_{k_2}\bigcup \left(\bigcup_{j\in\mathbb Z_p}E_{j,k_1}\right)\subseteq \Z{S}.$$
Recall the assumption that $|S|=|A|=p^2$, applying Lemma \ref{LmPL} on each $m\in M$, we will actually have 
$$\bigcup_{j\in\mathbb Z_p}E_{j,m}\subseteq \Z{S},$$ 
which combined with \eqref{EqSTP2M2}, \eqref{EqAnn} and Lemma \ref{LmGen} gives
$$(\G{p^2})\setminus \bigcup_{m'\in M'}K_{m'}\subseteq \Z{S}.$$
Now we can take an arbitrary $C_{j,m}$ with $j\in\mathbb Z_p, m\in M$, since 
$$\bigcup_{m'\in M'}K_{m'}\setminus\{(0,0)\}\subseteq K_m^{\perp_s}\setminus H_{j,m}=\Z{C_{j,m}},$$
we get
$$\Delta(\G{p^2})\subseteq Z(\F{S}\cdot \F{C_{j,m}}),$$
which similar as in \eqref{EqSTP2} means $|S|\ge p^3$, and is thus a contradiction.

\textbf{(2.2)} If $|M'|=0$, then $|M|=p+1$, i.e., $\Delta A\cap\Delta K_k\neq\emptyset$ holds for all $k\in\{0,1,\ldots,p-1,\infty\}$, by Lemma \ref{LmGen} this implies that 
\begin{equation} \label{EqSTP2_0}
\Delta K\subseteq \Z{S}.
\end{equation} 
Since we have assumed $|S|=p^2$,  \eqref{EqSTP2_0} together with \eqref{EqAnn} indicates that $S$ is a tiling complement of $K$, consequently by Lemma \ref{LmKC} we can conclude that $(A,S)$ is actually also a tiling pair.
\end{proof}

A subset $A$ in $\G{n}$ is called \emph{periodic} if $A=H\oplus A'$ for some subgroup $H$. A tiling pair is called periodic if one of the component is periodic. There is a famous result in \cite{sands1957,sands1962} that characterizes periodic factorizations in finite Abelian groups. In the recent paper \cite{fan2024} it was asked whether one can replace a component in a non-periodic tiling pair with a periodic one, such a property is formulated as the weak periodic tiling property therein. For $\G{p^2}$ we can prove the following:

\begin{Cor}
Let $(A,B)$ be a tiling pair in $\G{p^2}$, if neither of $A,B$ is periodic, then one of them can actually be replaced by a periodic set, i.e, either there is some periodic $A'$ so that $(A',B)$ is still a tiling pair, or there is some periodic $B'$ so that $(A,B')$ is still a tiling pair.
\end{Cor}

\begin{proof}
Without loss of generality let us just assume both $A,B$ contain the identity and $|A|\ge |B|$, then there are two cases:

\textbf{(1)} $|A|=p^3$ and $|B|=p$.

By Theorem \ref{ThmMain} and Lemma \ref{LmSEquiv} we can also view $B$ as symplectic spectral sets, then Lemma \ref{LmSTp} shows that $B$ admits a periodic tiling complement.

\textbf{(2)} $|A|=|B|=p^2$.

If $\Z{A}\cap  \Delta K_k\neq\emptyset$ for all $k\in\{0,1,\ldots,p-1,\infty\}$, then by Lemma \ref{LmZ} \ref{Z2} we can take $B'=K$.

Similarly if there is some $K_{k}$ and $E_{j,k}$ such that $\Z{A}\cap \Delta K_{k}\neq\emptyset$ and $\Z{A}\cap E_{j,k}\neq\emptyset$, then by Lemma \ref{LmZ} \ref{Z2} we can take $B'=H_{j,k}$.

Alternatively if there is some $K_{k}$ and $E_{j,k}$ such that $\Z{A}\cap\Delta K_k=\Z{A}\cap E_{j,k}=\emptyset$, then by \eqref{EqDecomp} we must have $\Delta H_{j,k}\subseteq \Z{B}$, thus by Lemma \ref{LmZ} \ref{Z2} we can take $A'=H_{j,k}$.

Finally we are again in the situation that we can partition the set $\{0,1,\ldots,p-1,\infty\}$ into $M$ and $M'$, so that
$$\tilde M=\{m: \Z{A}\cap\Delta K_m\neq \emptyset\}, \quad \tilde M'=\{m': \Z{A}\cap\Delta K_{m'}=\emptyset\}.$$
They are both non-empty, and for each $m\in\tilde M$, each $m'\in\tilde M'$ and each $j,j'\in\mathbb Z_p$ we have
$$\Z{A}\cap E_{j, m}=\emptyset, \quad \Z{A}\cap E_{j',m'}\neq\emptyset.$$
Again Theorem \ref{ThmMain} and Lemma \ref{LmSEquiv} shows that $A$ is a symplectic spectral set, then this is just Case (2.1) in the proof of Theorem \ref{ThmMain} (ii). As argued there a periodic $B'$ can be constructed using Lemma \ref{LmP2} \ref{P2ST}.
\end{proof}

{\small\noindent 
\textbf{Acknowledgement:} The author would like to thank Aihua Fan(CCNU/WHU), Shilie Fan(CCNU), Lingmin Liao(WHU), Ruxi Shi (FDU), Tao Zhang(XDU) for a fruitful discussion and many helpful comments on the tiling to spectral part of this paper. There have been many efforts on improving this paper, the author thanks all the anonymous reviewers for spending considerable time on this paper and providing valuable inputs.

The author thanks all anonymous reviewers of this paper for the patience and many helpful suggestions.
}

\end{document}